%% file: Z2inh2M3.tex
\documentclass[12pt]{amsart}
\input{xypic}
\usepackage{amsmath}
\usepackage{amssymb}
\usepackage{latexsym}
\usepackage{enumerate}
\usepackage[active]{srcltx}
\usepackage{graphics}
\usepackage{epsfig}

\newtheorem{theorem}{Theorem}[section]

\newtheorem{lemma}[theorem]{Lemma}

\newcommand{\Z}{{\mathbb Z}}

\newcommand{\downarrowright}[1]{\downarrow
\rlap{\raise0.1cm\hbox{$\scriptstyle{#1}$}}}

\newcommand{\downarrowleft}[1]{\rlap{\kern-0.2cm
\raise0.1cm\hbox{$\scriptstyle{#1}$}}\downarrow}

\newcommand{\uparrowright}[1]{\uparrow
\rlap{\lower0.1cm\hbox{$\scriptstyle{#1}$}}}

\newcommand{\uparrowleft}[1]{\rlap{\kern-0.2cm
\lower0.1cm\hbox{$\scriptstyle{#1}$}}\uparrow}

\def\umono{\ar@{_{(}->}[u]}
\def\uumono{\ar@{_{(}->}[uu]}

\def\lmono{\ar@{_{(}->}[l]}
\def\llmono{\ar@{_{(}->}[ll]}

\begin{document}

\title{The 2-torsion in the second homology of the genus $3$ mapping class group}

 \date{26 October,  2013}
\author{Wolfgang Pitsch}
\address{Universitat Aut\`onoma de Barcelona \\ Departament de Matem\`atiques. E-08193 Bellaterra, Spain}
\email{pitsch@mat.uab.es}

\thanks{The author was partially supported  by FEDER/MEC grant MTM2010-20622.} 



%%%%%%%%%%%%%%%%%%%%%%%%%%%%%%%%%%%%%%%%%%%%%%%%%%%%%%%%%%%%%%%%%
%                         ABSTRACT                              %
%%%%%%%%%%%%%%%%%%%%%%%%%%%%%%%%%%%%%%%%%%%%%%%%%%%%%%%%%%%%%%%%%
\begin{abstract}
 This work is NOT to be used as reference. First, because as C.F.~B\"odigheimer and M.~Korkmaz pointed to us  the computation of the $\Z_2$ factor that remained undecided in M.~Korkmaz and A. Stipsicz, {\em The second homology groups of mapping class groups of orientable surfaces.} Math. Proc. Camb. Phil. Soc., was shown to exist by Skasai, see hi Theorem 4.9 and Corollary 4.10 in {\em Lagrangian mapping class groups from a group homological point of view.} Algebr. Geom. Topol. 12 (2012), no. 1, 267--291. Second, because one could obtain this result by gathering old results in the literature,  first by noticing as Korkmaz kindly reminded me, that D.~Johnson, in  \emph{Homeomorphisms of a surface which act trivially on homology} Porc. AMS Volume 75, Number 1, 1979.  proved that the quotient of the Torelli group $\mathcal{T}_g/[\mathcal{T}_g,\mathcal{M}_g]$ is trivial for $g\geq 3$,  the five term exact sequence then implies that the $\Z_2$ factor in Stein's computation of $H_2(Sp(6,\mathbf{Z});\mathbf{Z}) = \Z \oplus\Z_2$ (see his {\em The Schur Multipliers of $Sp_6(\mathbf{Z}), Spin_8(\mathbf{Z}), Spin_7(\mathbf{Z}),$ and $F_4(\mathbf{Z})$.} Math. Ann. 215 (1975), 173--193. ), detects the undecided $\Z_2$ factor in $H_2(\mathbf{M}_3;\Z)$.
 
 \end{abstract}

 \maketitle

%%%%%%%%%%%%%%%%%%%%%%%%%%%%%%%%%%%%%%%%%%%%%%%%%%%%%%%%%%%%%%%%%
%                       INTRODUCTION                            %
%%%%%%%%%%%%%%%%%%%%%%%%%%%%%%%%%%%%%%%%%%%%%%%%%%%%%%%%%%%%%%%%%

%%%%%%%%%%%%%%%%%%%%%%%%%%%%%%%%%%%%%%%%%%%%%%%%%
%%%%%%%%%%%%%%%%%%%%%%%%%%%%%%%%%%%%%%%%%%%%%%%%%
\section{Introduction}
%%%%%%%%%%%%%%%%%%%%%%%%%%%%%%%%%%%%%%%%%%%%%%%%%
%%%%%%%%%%%%%%%%%%%%%%%%%%%%%%%%%%%%%%%%%%%%%%%%%

Denote by $\Sigma_{g,n}^r$  an oriented surface of genus $g$ with $n$ boundary components and $r$ punctures and by $\mathcal{M}_{g,n}^r$ its mapping class group, that is the group of isotopy classes of orientation-preserving diffeomorphisms of $\Sigma_{g,n}^r$ that are the identity on the boundary and fix the punctures.  Also denote by $\mathbf{Z}_2$ the mod $2$ reduction of $\mathbf{Z}$. In a famous paper \cite{Harer} Harer computed the second homology group  of mapping class group for $g \geq 5$. Then, using a presentation of $\mathcal{M}_{g,1}$ given by Wajnryb in \cite{Wajnryb} we showed in \cite{Pitsch} that Harer's computations could be obtained from Hopf's formula and extended to $g \geq 4$, yielding moreover an explicit generator for this group. Later on in \cite{Korkmaz} Korkmaz and Stipsicz pushed this computations to encompass the remaining genus $g=2,3$ and $n \geq 2, r \geq 1$. Notice that for $g=2$ Benson and Cohen had computed the Poincar\'e series of $H_\ast(\mathcal{M}_2;\mathbf{Z}_2)$.  
Unfortunately a small gap remained after Korkmaz and Stipsicz computations: they showed that  $H_2(\mathcal{M}_3;\mathbf{Z}) \simeq \mathbf{Z} \oplus A $ and  $H_2(\mathcal{M}_{3,1};\mathbf{Z}) \simeq \mathbf{Z} \oplus B$, where $0 \leq B \leq A \leq \mathbf{Z}_2$. The purpose of this note is to close this gap and to finally prove:

%%%%%%%%%%%%%%%%%%%%%%%%%%%%%%%%%%%%%%%%%%%%%%%%%
\begin{theorem}
%%%%%%%%%%%%%%%%%%%%%%%%%%%%%%%%%%%%%%%%%%%%%%%%%
We have $H_2(\mathcal{M}_{3}; \mathbf{Z}) \simeq \mathbf{Z}\oplus \mathbf{Z}/2$ and $H_2(\mathcal{M}_{3,1}; \mathbf{Z}) \simeq \mathbf{Z}\oplus \mathbf{Z}/2.$
\end{theorem}
%%%%%%%%%%%%%%%%%%%%%%%%%%%%%%%%%%%%%%%%%%%%%%%%%

As observed by Korkmaz and Stipsicz the computations using Hopf's formula show that the  homomorphism  induced by capping-off the boundary component $H_2(\mathcal{M}_{3,1}; \mathbf{Z})\rightarrow H_2(\mathcal{M}_{3}; \mathbf{Z})$ is surjective and hence, from our computation, an isomorphism.

%%%%%%%%%%%%%%%%%%%%%%%%%%%%%%%%%%%%%%%%%%%%%%%%%
%%%%%%%%%%%%%%%%%%%%%%%%%%%%%%%%%%%%%%%%%%%%%%%%%
\section{Proof of the Theorem}
%%%%%%%%%%%%%%%%%%%%%%%%%%%%%%%%%%%%%%%%%%%%%%%%%
%%%%%%%%%%%%%%%%%%%%%%%%%%%%%%%%%%%%%%%%%%%%%%%%%

Denote by $\mathcal{T}_{3}$ or $\mathcal{T}_{3,1}$ accordingly the Torelli groups, that is the kernel of the surjective map from the mapping class group onto the  symplectic group with integer entries $Sp(6,\mathbf{Z})$. By computations of Stein \cite{Stein} we know that $H_2(\mathrm{Sp}(6;\mathbf{Z});\mathbf{Z}) \simeq \mathbf{Z}\oplus \mathbf{Z}/2$. Since mapping class groups of genus $g \geq 3$ are perfect, the $5$ term exact sequence in low dimensional homology gives a small exact sequence:
\[
\xymatrix{
 H_2(\mathcal{M}_{\ast};\mathbf{Z}) \ar[r] & H_2(\mathrm{Sp}(6;\mathbf{Z});\mathbf{Z}) \ar[r] & (H_1(\mathcal{T}_\ast;\mathbf{Z}))_{\mathrm{Sp}(6;\mathbf{Z})} \ar[r] & 0 
}
\]
Where $\ast$ denotes either $g$ or $g,1$ and in view of Stein's result all we have to do is to show that $H_1(\mathcal{T}_\ast;\mathbf{Z}))_{\mathrm{Sp}(6;\mathbf{Z})} \simeq {\mathcal{T}}_{\ast}/[\mathcal{T}_{\ast},\mathcal{M}_{\ast}] = 0$.

By Johnson's fundamental result the Torelli group $\mathcal{T}_{3}$ and $\mathcal{T}_{3,1}$ are generated by twists along bounding pairs of genus $1$, that is by mapping classes of the form $T_\alpha T_\beta^{-1}$, where $\alpha$ and $\beta$ are two simple closed curves that are not isotopic, homologous, not homologous to $0$ and such that the complement of $\{ \alpha, \beta\}$ has a component of genus $1$. In particular capping-off the boundary component induces a surjective map $H_1(\mathcal{T}_{g,1};\mathbf{Z}) \rightarrow H_1(\mathcal{T}_{g};\mathbf{Z})$ and since taking coinvariants is a right-exact functor it is enough to prove that
\[
H_1(\mathcal{T}_{g,1};\mathbf{Z})_{\mathrm{Sp}(6;\mathbf{Z})} \simeq \mathcal{T}_{g,1}/[\mathcal{T}_{g,1},\mathcal{M}_{g,1}] = 0.
\]

The mapping class group acts transitively on bounding pairs of genus $1$, hence the group $ \mathcal{T}_{g,1}/[\mathcal{T}_{g,1},\mathcal{M}_{g,1}]$ is monogenic generated by the class of any bounding pair map. Also, given a bounding pair $\{ \alpha, \beta\}$ there exists a mapping class, say $\phi$, that exchanges $\alpha$ and $\beta$, hence in $ \mathcal{T}_{g,1}/[\mathcal{T}_{g,1},\mathcal{M}_{g,1}]$:
\[
T_\alpha {T_\beta}^{-1} = \phi T_\alpha {T_\beta}^{-1} \phi^{-1} = T_{\phi(\alpha)} {T_{\phi(\beta)}}^{-1} = T_{\beta}{T_{\alpha}}^{-1} =  (T_\alpha T_\beta^{-1})^{-1},
\] 
and  $\mathcal{T}_{g,1}/[\mathcal{T}_{g,1},\mathcal{M}_{g,1}]$ is at most $\mathbf{Z}_2$.

In \cite{JohnIII} Johnson computed the abelianization of the Torelli group $\mathcal{T}_{g,1}$, it fits into a short exact sequence:
\[
\xymatrix{
0 \ar[r] & B^2_{3,1} \ar[r] & H_1(\mathcal{T}_{g,1}; \mathbf{Z}) \ar[r] & \Lambda^3 H \ar[r] & 0, 
}
\]
where:
\begin{enumerate}
\item $H$ stands for the homology group $\mathrm{H}_1(\Sigma_{3,1};\mathbf{Z})$,
\item $B^2_{3,1}$ is the Boolean algebra of polynomials of degree $\leq 2$ generated by the elements $\overline{x} \in H$ and subject to the relations:
\begin{itemize}
\item $\overline{x + y} = \overline{x} + \overline{y} + x\cdot y$, where $x\cdot y$ is the mod $2$ intersection number,
\item $\overline{x}^2 = \overline{x}$.
\end{itemize}
\end{enumerate}

Notice that this is a short exact sequence of $\mathrm{Sp}(6;\mathbf{Z})$-modules, where the action on the quotient is simply given by the third exterior power of the action on homology and the action of the kernel is given on generators by $\phi(\overline{x}) = \overline{\phi(x)}$ extended in the obvious way. Finally, this kernel is a $\mathbf{Z}_2$-vector space and is the image in $H_1(\mathcal{T}_{g,1};\mathbf{Z})$ of the Johnson subgroup $\mathcal{K}_{g,1}$,  the subgroup generated by twists along bounding simple closed curves.

%%%%%%%%%%%%%%%%%%%%%%%%%%%%%%%%%%%%%%%%%%%%%%%%%
\begin{lemma}
%%%%%%%%%%%%%%%%%%%%%%%%%%%%%%%%%%%%%%%%%%%%%%%%%
The image of $B^2_{3,1}$ in $H_1(\mathcal{T}_{g,1};\mathbf{Z})_{\mathrm{Sp}(6;\mathbf{Z})}$ is trivial.
\end{lemma}
%%%%%%%%%%%%%%%%%%%%%%%%%%%%%%%%%%%%%%%%%%%%%%%%%
\begin{proof}
This image is a quotient of $(B^2_{3,1})_{\mathrm{Sp}(6;\mathbf{Z})}$, so it suffices to prove that this group is trivial. First notice that since the mod $2$ symplectic form is non-degenerated any non-zero element in $H$ can be completed into a symplectic basis, and in particular $\mathrm{Sp}(6;\mathbf{Z})$ acts transitively on the non-zero elements in $H$. Let $a,b$ be two elements such that $a\cdot b =1$, then if $\tau_a$ denotes the transvection along $a$, we have $\tau_a(\overline{b})= \overline{a +b} = \overline{a} + \overline{b} +1$, and in  $(B^2_{3,1})_{\mathrm{Sp}(6;\mathbf{Z})}$ this gives that $\overline{a} = 1$, hence in $(B^2_{3,1})_{\mathrm{Sp}(6;\mathbf{Z})}$ all monomials of degree $1$ or $2$ are in fact constants and this group is at most $\mathbf{Z}_2$. Finally, let $c \in H$ be such that $a\cdot c = 0 = b \cdot c$. In  $(B^2_{3,1})_{\mathrm{Sp}(6;\mathbf{Z})}$ we have:
\[
\begin{array}{rcccl}
 1 & = & \overline{a}\overline{b} & = & \tau_{b+c}(\overline{a} \overline{b}) \\
 & = & \overline{a + b+ c} \ \overline{b} & =&  (\overline{a} + \overline{b} + \overline{c} +1)\overline{b} \\ 
 & = & \overline{a}\overline{b} + \overline{b}^2 + \overline{c}\overline{b} +\overline{b} & = &  1 + 1  +1 +1  \\
 & =&  0. &&
\end{array}
\]
\end{proof}

To conclude  we apply the lantern relation:

\begin{center}
\input{Rellanterne.pspdftex}
\end{center}

to the following curves (we only draw the four boundary curves).
\begin{center}
\input{lanterne.pspdftex}
\end{center}
The lantern relation  shows that the twist around $\beta_0$, a bounding simple closed curve of genus $2$ can be written as the product of three  twists around bounding pairs of genus $1$:
\[
T_{\beta_0}=T_{\gamma_{12}}T_{\beta_3}^{-1} T_{\gamma_{13}}T_{\beta_2}^{-1}T_{\gamma_{23}}T_{\beta_1}^{-1}
\]
If $t$ denotes the generator of $H_1(\mathcal{T}_{g,1};\mathbf{Z})_{\mathrm{Sp}(6;\mathbf{Z})}$, which is of order $2$, then this equation becomes $0 = t^3$, and this finishes the proof.

\end{document}

%% file: Rellanterne.pspdftex
\begin{picture}(0,0)%
\includegraphics{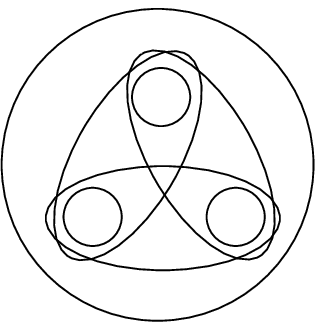}%
\end{picture}%
\setlength{\unitlength}{4144sp}%
\begingroup\makeatletter\ifx\SetFigFont\undefined%
\gdef\SetFigFont#1#2#3#4#5{%
  \reset@font\fontsize{#1}{#2pt}%
  \fontfamily{#3}\fontseries{#4}\fontshape{#5}%
  \selectfont}%
\fi\endgroup%
\begin{picture}(1442,1903)(180,-1104)
\put(581,-336){\makebox(0,0)[b]{\smash{{\SetFigFont{8}{9.6}{\rmdefault}{\mddefault}{\updefault}$\beta_1$}}}}
\put(910,224){\makebox(0,0)[b]{\smash{{\SetFigFont{8}{9.6}{\rmdefault}{\mddefault}{\updefault}$\beta_2$}}}}
\put(1281,-357){\makebox(0,0)[b]{\smash{{\SetFigFont{8}{9.6}{\rmdefault}{\mddefault}{\updefault}$\beta_3$}}}}
\put(936,700){\makebox(0,0)[b]{\smash{{\SetFigFont{8}{9.6}{\rmdefault}{\mddefault}{\updefault}$\beta_0$}}}}
\put(1373, 65){\makebox(0,0)[lb]{\smash{{\SetFigFont{8}{9.6}{\rmdefault}{\mddefault}{\updefault}$\gamma_{23}$}}}}
\put(918,-672){\makebox(0,0)[b]{\smash{{\SetFigFont{8}{9.6}{\rmdefault}{\mddefault}{\updefault}$\gamma_{13}$}}}}
\put(446, 65){\makebox(0,0)[rb]{\smash{{\SetFigFont{8}{9.6}{\rmdefault}{\mddefault}{\updefault}$\gamma_{12}$}}}}
\put(967,-1040){\makebox(0,0)[b]{\smash{{\SetFigFont{10}{12.0}{\rmdefault}{\mddefault}{\updefault}$T_{\beta_0} T_{\beta_1} T_{\beta_2} T_{\beta_3} = T_{\gamma_{12}}T_{\gamma_{13}} T_{\gamma_{23}}$}}}}
\end{picture}%

%% file: lanterne.pspdftex
\begin{picture}(0,0)%
\includegraphics{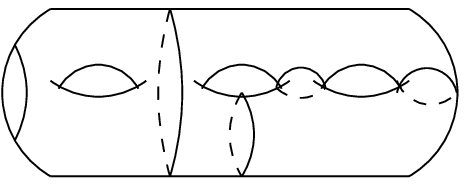}%
\end{picture}%
\setlength{\unitlength}{4144sp}%
\begingroup\makeatletter\ifx\SetFigFont\undefined%
\gdef\SetFigFont#1#2#3#4#5{%
  \reset@font\fontsize{#1}{#2pt}%
  \fontfamily{#3}\fontseries{#4}\fontshape{#5}%
  \selectfont}%
\fi\endgroup%
\begin{picture}(2133,1090)(268,-737)
\put(1036,254){\makebox(0,0)[b]{\smash{{\SetFigFont{8}{9.6}{\rmdefault}{\mddefault}{\updefault}$\beta_0$}}}}
\put(1666,-331){\makebox(0,0)[b]{\smash{{\SetFigFont{8}{9.6}{\rmdefault}{\mddefault}{\updefault}$\beta_2$}}}}
\put(1396,-691){\makebox(0,0)[b]{\smash{{\SetFigFont{8}{9.6}{\rmdefault}{\mddefault}{\updefault}$\beta_1$}}}}
\put(2386,-196){\makebox(0,0)[lb]{\smash{{\SetFigFont{8}{9.6}{\rmdefault}{\mddefault}{\updefault}$\beta_3$}}}}
\end{picture}%